\documentclass[10pt]{elsarticle}
\usepackage[cp1251]{inputenc}
\usepackage[english]{babel}
\usepackage{amsmath}
\usepackage{amssymb}
\usepackage{amsfonts}
\usepackage{amsthm}
\usepackage{mathtools}
\usepackage{dsfont}
\usepackage{rotating}
\usepackage{graphicx}
\usepackage{floatflt,epsfig}
\usepackage{lineno,hyperref}
\usepackage{enumerate}
\usepackage{colortbl}
\usepackage{array,tabularx,tabulary,booktabs}
\usepackage{longtable}
\usepackage{multirow}
\usepackage{wrapfig}
\usepackage{subcaption}
\usepackage{pdflscape}
\usepackage[table]{xcolor}
\newcolumntype{^}{>{\currentrowstyle}}

\journal{Arxiv}
\newtheorem{lemma}{Lemma}

\newtheorem{theorem}{Theorem}
\newtheorem{corollary}{Corollary}
\newtheorem{proposition}{Proposition}

\newtheorem{problem}{Problem}

\begin{document}
\renewcommand{\abstractname}{Abstract}
\renewcommand{\refname}{References}
\renewcommand{\tablename}{Table}
\renewcommand{\arraystretch}{0.9}
\thispagestyle{empty}
\sloppy

\begin{frontmatter}
\title{On eigenfunctions of the block graphs of geometric Steiner systems}

\author[01]{Sergey Goryainov}
\ead{sergey.goryainov3@gmail.com}

\author[02]{Dmitry Panasenko}
\ead{makare95@mail.ru}

\address[01] {School of Mathematical Sciences, Hebei International Joint Research Center for Mathematics and Interdisciplinary Science, Hebei Key Laboratory of Computational Mathematics and Applications\\Hebei Normal University, Shijiazhuang  050024, P.R. China}
\address[02] {Chelyabinsk State University, 129 Bratiev Kashirinykh st.,\\Chelyabinsk 454021, Russia}


\begin{abstract}
This paper lies in the context of the studies of eigenfunctions of graphs having minimum cardinality of support. One of the tools is the weight-distribution bound, a lower bound on the cardinality of support of an eigenfunction of a distance-regular graph corresponding to a non-principal eigenvalue.
The tightness of the weight-distribution bound was previously shown in general for the smallest eigenvalue of a Grassmann graph. However, a characterisation of optimal eigenfunctions was not obtained. Motivated by this open problem, we consider the class of strongly regular Grassmann graphs and give the required characterisation in this case. We then show the tightness of the weight-distribution bound for block graphs of affine designs (defined on the lines of an affine space with two lines being adjacent when intersect) and obtain a similar characterisation of optimal eigenfunctions. 
\end{abstract}

\begin{keyword}
eigenfunction; weight-distribution bound; projective space; affine space; equitable 2-partition
\vspace{\baselineskip}
\MSC[2020] 05B05 \sep 05B25 \sep 05C25
\end{keyword}
\end{frontmatter}

\section{Introduction}\label{Intro}
This paper lies in the context of the studies of eigenfunctions of graphs having minimum cardinality of support (see \cite{SV21} for the background). One of the tools is the weight-distribution bound, a lower bound on the cardinality of support of an eigenfunction of a distance-regular graph corresponding to a non-principal eigenvalue.
Note that the tightness of the weight-distribution bound was previously shown in general for the smallest eigenvalue of a Grassmann graph $J_q(n+1,e)$ (see \cite[Section 4]{KMP16}). However, a characterisation of optimal eigenfunctions was not obtained. Motivated by this open problem, we consider the class of strongly regular Grassmann graphs and give the required characterisation in this case.
\begin{theorem}\label{CharacterisationGrassmann}
There is a one-to-one correspondence between optimal $-(q+1)$-eigenfunctions of the Grassmann graph $J_q(n+1,2)$ and reguli in 3-dimensional subspaces of $\operatorname{PG}(n,q)$.    
\end{theorem}

Let $X_q(n,1)$ denote the block graph of an affine design $AS(n,q)$, $n \ge 3$. The following theorem shows that the weight-distribution bound is tight for the negative eigenvalue $-q$ of $X_q(n,1)$ and provides a geometric characterisation of optimal eigenfunctions.
\begin{theorem}\label{CharacterisationAffine}
The weight-distribution bound is tight for the eigenvalue $-q$ of $X_q(n,1)$. Moreover, there are exactly two types of optimal $-q$-eigenfunctions: there is a one-to-one correspondence between optimal eigenfunctions of the first type and pairs of parallel classes of lines from some subplane $\operatorname{AG}(2,q)$; there is a one-to-one correspondence between optimal eigenfunctions of the second type and affine reguli in 3-dimensional subspaces of $\operatorname{AG}(n,q)$.     
\end{theorem}

Similar to \cite[Theorem 2.4.4]{BR98}, the following proposition gives an explicit general algebraic description of an affine regulus in $\operatorname{AG}(n,q)$.

\begin{proposition}\label{ConstructionOfAffineRegulus}
Let $v_1,v_2,v_3$ be three linearly independent vectors from an $n$-dimensional vector space over a finite field $\mathbb{F}_q$, $n \ge 3$. Let $$S_1 = \{\langle cv_3 + v_1 \rangle + cv_2 \; | \; c \in \mathbb{F}_q\} \text{ and } S_2 = \{\langle cv_3 + v_2 \rangle + cv_1 \; | \; c \in \mathbb{F}_q\}.$$ 
Then the following statements hold.\\
{\rm(1)} The sets $S_1$ and $S_2$ are affine reguli opposite to each other in the subspace $\operatorname{AG}(3,q)$ of $\operatorname{AG}(n,q)$ generated by the vectors $v_1,v_2,v_3$.\\
{\rm(2)} Each of the reguli $S_1$ and $S_2$ can be embedded into a class of parallel planes in the subspace $\operatorname{AG}(3,q)$ generated by $v_1,v_2,v_3$.
\end{proposition}
Note that there are exactly $q^4(q^3-1)(q+1)$ different affine reguli in $\operatorname{AG}(3,q)$.

Let $R = \{\ell_1,\ell_2,\ldots, \ell_{q+1}\}$ and $R^{opp} = \{m_1,m_2,\ldots, m_{q+1}\}$ be a regulus and its opposite in $\operatorname{PG}(3,q)$ embedded into $\operatorname{PG}(n,q)$. Let $H$ be a hyperplane in $\operatorname{PG}(n,q)$ not containing any line from $R \cup R^{opp}$. Let $R' = \{\ell_1', \ell_2', \ldots, \ell_{q+1}'\}$ and $R^{opp} = \{m_1', m_2', \ldots, m_{q+1}'\}$ be affine lines in $\operatorname{PG}(n,q) \setminus H$ obtained from the lines $\ell_1,\ell_2,\ldots, \ell_{q+1}$ and $m_1,m_2,\ldots, m_{q+1}$, respectively, by removing a point at infinity. Consider the block graph $Y_q$ of the affine design on the points and lines of $\operatorname{PG}(n,q) \setminus H$, isomorphic to $X_q(n,1)$. Define a function $f:V(Y_q) \mapsto \mathbb{R}$ by the following rule: for any $\ell \in V(Y_q)$, put
$$
f(\ell) = 
\left\{
\begin{array}{rl}
        1, & \text{for } \ell \in R';\\
        -1, & \text{for } \ell \in (R^{opp})';\\
        0, & \text{otherwise.}
\end{array}
\right.
$$

\begin{proposition}\label{WDBPlus2}
The function $f$ is a $-q$-eigenfunction of $Y_q$. The support of $f$ induces a complete bipartite subgraph with a removed perfect matching.  
\end{proposition}
Note that the cardinality of support of $f$ is $2(q+1)$, which is larger by 2 than the value $2q$ given by the weight-distribution bound. The following problem naturally arises.

\begin{problem}\label{prob:WDBPlus2}
What are the eigenfunctions of $X_q(n,1)$ whose cardinality of support is $2(q+1)$, that is, larger by 2 than the value $2q$ given by the weight-distribution bound?    
\end{problem}

In this paper we are not only interested in optimal eigenfunctions. We also discuss an approach for studying equitable 2-partitions based on $(1,-1,0)$-valued eigenfunctions. 
It is well-known that the investigation of equitable $2$-partitions is equivalent to the investigation of eigenfunctions
with two values (see Lemma \ref{1to1}, for example). 
Since eigenfunctions corresponding to distinct eigenvalues of a graph are orthogonal,
one can use $\theta_1$-eigenfunctions of a regular graph $G$ to restrict possible $\theta_2$-equitable $2$-partitions
of the graph $G$ whenever $\theta_1 \ne \theta_2$ (see \cite{T20} for a general discussion).

Let $f_1$ be a $(1,-1,0)$-valued function that is a linear combination of eigenfunctions of $G$ corresponding to non-principal eigenvalues. Denote by $Supp^+(f_1)$ and $Supp^-(f_1)$
the sets of vertices of $G$ where $f_1$ takes values $1$ and $-1$, respectively. 
Since $f_1$ is orthogonal to the constant $1$-valued function, which is a $k$-eigenfunction of $G$ (where $k$ is the degree of $G$),
we have $|Supp^+(f_1)| = |Supp^-(f_1)|$.

Take a non-principal eigenvalue $\theta$ of $G$ such that each of the eigenfunctions involved in the definition of $f_1$ does not correspond to $\theta$.
Consider an arbitrary $\theta$-equitable $2$-partition $\Pi = (V_1,V_2)$ of $G$ and consider the
$\theta$-eigenfunction $f_2$ of $G$ representing $\Pi$, that is,
the function $f_2$ is $(x_1,x_2)$-valued, where $(x_1,x_2)^T$ is a $\theta$-eigenvector of $P_\Pi$, $x_1 \ne x_2$, and $\Pi_{(x_1,x_2)} = \Pi$.
In order to formulate the following proposition, we introduce the $V_1$-indicator function
as follows:
$$
\overline{v} :=
\left\{
  \begin{array}{ll}
    1, & v \in V_1; \\
    0, & v \in V_2.
  \end{array}
\right.
$$
and, for a set of of vertices $U$ put $\overline{U}:=\sum\limits_{u \in U}\overline{u}$.

\begin{proposition}\label{balanceTheorem}
The following equality holds
\begin{equation}\label{balance}
\overline{Supp^+(f_1)} =  \overline{Supp^-(f_1)}.
\end{equation}
\end{proposition}

\medskip
\begin{corollary}\label{Balance}
Proposition \ref{balanceTheorem} shows that the sets $Supp^+(f_1)$ and $Supp^-(f_1)$ have the same number of vertices from $V_1$.    
\end{corollary}

Let $\theta$ be the positive non-principal eigenvalue of a Grassmann graph $J_q(4,2)$. A part of a $\theta$-equitable 2-partition of $J_q(4,2)$ is called a \emph{Cameron-Liebler line class}.
It is well-known that one of the equivalent definitions of a Cameron-Liebler line class is as follows (see \cite[Theorem 1(v)]{P91}): a class of lines $L$ in $\operatorname{PG}(3,q)$ is Cameron-Liebler if and only if $|L \cap R| = |L \cap R^{opp}|$ for all reguli $R$ in $\operatorname{PG}(3,q)$. We note that this condition naturally follows from Theorem \ref{CharacterisationGrassmann} and Corollary \ref{Balance}. Cameron-Liebler classes have been intensively studied (see, for example, \cite{GM14} and \cite{FMRXZ21}). 
Cameron-Liebler lines classes in affine spaces have been recently studied in \cite{DMSS20} and \cite{DIMS21}.

This paper is organised as follows. In Section \ref{prelim}, we list some preliminary definitions and results. In Section \ref{ConstructionsAndCharacterisation}, we prove Theorem \ref{CharacterisationGrassmann}, Theorem \ref{CharacterisationAffine}, Proposition \ref{ConstructionOfAffineRegulus} and Proposition \ref{WDBPlus2}. In Section \ref{sec:balance}, we prove Proposition \ref{balanceTheorem} and Corollary \ref{Balance}.

\section{Preliminaries}\label{prelim}
In this section we list some preliminary definitions and results.

\subsection{The weight-distribution bound for strongly regular graphs} \label{wdb}
A $k$-regular graph on $v$ vertices is called \emph{strongly regular} with parameters $(v,k,\lambda,\mu)$ if any two adjacent vertices have $\lambda$ common neighbours and any two distinct non-adjacent vertices have $\mu$ common neighbours. 
If $G$ is a strongly regular graph, then its complement is also a strongly regular
graph. A strongly regular graph $G$ is \emph{primitive} if both $G$ and its complement
are connected. If $G$ is not primitive, we call it \emph{imprimitive}. The imprimitive
strongly regular graphs are exactly the disjoint unions of complete graphs and
their complements, namely, the complete multipartite graphs with parts of the same size. We focus on primitive strongly regular graphs.

\begin{lemma}[{\cite[Theorem 5.2.1]{GM15}}]\label{EigenvaluesSRG}
If $G$ is a primitive strongly regular graph with parameters $(v,k,\lambda,\mu)$ and  
$$
\Delta:=\sqrt{(\lambda-\mu)^2+4(k-\mu)},
$$
then $G$ has exactly three eigenvalues 
$$
k,~~~r = \frac{\lambda-\mu+\Delta}{2},~~~s = \frac{\lambda-\mu-\Delta}{2},
$$
with respective multiplicities
$$
m_k = 1,~~~m_{r} = -\frac{(v-1)s+k}{r-s},~~~
m_{s} = \frac{(v-1)r+k}{r-s}.
$$
\end{lemma}
Note that $s < 0 < r$ holds. For the primitive strongly regular graph $G$ from the above lemma, the matrix
$$
\left(
  \begin{array}{c|cc}
    1 & k & v-1-k\\
    m_{r} & r & -1-r \\
    m_{s} & s & -1-s\\
  \end{array}
\right).
$$
is called the \emph{modified matrix of eigenvalues}. The first column gives the dimensions of the eigenspaces (i.e., the multiplicities of the eigenvalues); the second column contains the eigenvalues of the graph, and the third gives the eigenvalues of its complement.

Let $\theta$ be an eigenvalue of a graph $G$. A real-valued function $f$ on the vertex set of $G$ is called an \emph{eigenfunction}
of the graph $G$ corresponding to the eigenvalue $\theta$ (or a \emph{$\theta$-eigenfunction} of $G$), if $f \not \equiv 0$ and
for any vertex $u$ in $G$ the condition
\begin{equation}\label{LocalCondition}
\theta\cdot f(u)=\sum_{\substack{w\in{G(u)}}}f(w)
\end{equation}
holds, where $G(u)$ is the set of neighbours of the vertex $u$.
Although eigenfunctions of graphs receive less attention of researchers in contrast to their eigenvalues, there are still plenty of related literature. We refer to the recent survey \cite{SV21} for a summary of results on the problem of finding the minimum cardinality of support of eigenfunctions of graphs and characterising the optimal eigenfunctions. For an eigenspace of a graph, an eigenfunction is called \emph{optimal} if it has the minimum cardinality of support among all eigenfunctions in this eigenspace.

The following lemma gives a lower bound for the number of non-zeroes (i.e., the cardinality of the support) for an eigenfunction of a strongly regular graph. In fact, this is a special case of a more general result for distance-regular graphs \cite[Section 2.4]{KMP16}. 
\begin{lemma}\label{WDB}
Let $G$ be a primitive strongly regular graph with parameters $(v,k,\lambda,\mu)$ and let $\theta$ be a non-principal eigenvalue of $G$. Then an eigenfunction of $G$ corresponding to the eigenvalue $\theta$ has at least $$1+|\theta|+\bigg|\frac{(\theta-\lambda)\theta-k}{\mu}\bigg|$$ non-zeroes.
\end{lemma}
\begin{proof}
If follows from \cite[Corollary 1]{KMP16} and the fact that a primitive strongly regular graph with parameters $(v,k,\lambda,\mu)$ is a distance-regular graph of diameter $2$ with $a_1 = \lambda$, $b_0 = k$ and $c_2 = \mu$. 
\end{proof}

\medskip

The lower bound given in Lemma~\ref{WDB} is known as the {\em weight-distribution bound}. Next we explicitly determine the weight-distribution bound in terms of the corresponding eigenvalue.

\begin{corollary}\label{WDBsrg}
Let $G$ be a primitive strongly regular graph with non-principal eigenvalues $s<0<r$. Then an eigenfunction of $G$ corresponding to the eigenvalue $r$ has at least $2(r+1)$ non-zeroes, and an eigenfunction corresponding to the eigenvalue $s$ has at least $-2s$ non-zeroes.
\end{corollary}

\begin{proof}
Suppose $G$ is strongly regular with parameters $(v,k,\lambda,\mu)$. Note that $s,r$ are the two solutions of the quadratic equation $x^2-(\lambda-\mu)x-(k-\mu)=0$ according to Lemma~\ref{EigenvaluesSRG}. Thus, $(s-\lambda)s-k=-\mu(s+1)$ and $(r-\lambda)r-k=-\mu(r+1)$. It follows that
$$
1+|s|+\bigg|\frac{(s-\lambda)s-k}{\mu}\bigg|=1-s+\bigg|\frac{-\mu(s+1)}{\mu}\bigg|=1-s-(s+1)=-2s,
$$
$$
1+|r|+\bigg|\frac{(r-\lambda)r-k}{\mu}\bigg|=1+r+\bigg|\frac{-\mu(r+1)}{\mu}\bigg|=1+r+(r+1)=2(r+1).
$$
The proof finishes by applying Lemma~\ref{WDB}. 
\end{proof}

A clique in a strongly regular graph $G$ is called a \emph{Delsarte clique} if its size meets the Delsarte-Hoffman bound (see \cite[Proposition 1.3.2]{BCN89}).

The following lemma gives a combinatorial interpretation of the tightness of the weight-distribution bound for a non-principal eigenvalue of a strongly regular graph in terms of special induced subgraphs.

\begin{lemma}\label{OptimalEigenfunctionsSRG}
Let $G$ be a primitive strongly regular graph with non-principal eigenvalues $s,r$, where $s < 0 < r$. Then the following statements hold.\\
{\rm (1)} For an $s$-eigenfunction $f$, if the cardinality of support of $f$ meets the weight-distribution bound, then there exists an induced complete bipartite subgraph in $G$ with parts $T_0$ and $T_1$ of size $-s$. Moreover, up to multiplication by a constant, $f$ has value $1$ on the vertices of $T_0$ and value $-1$ on the vertices of $T_1$. \\
{\rm (2)} For an $r$-eigenfunction $f$, if the cardinality of support of $f$ meets the weight-distribution bound, then there exists an induced pair of isolated cliques $T_0$ and $T_1$ in $G$ of size $-\overline{s} = -(-1-r) = 1+r$. Moreover, up to multiplication by a constant, $f$ has value 1 on the vertices of $T_0$ and value $-1$ on the vertices of $T_1$.\\
{\rm (3)} If $G$ has Delsarte cliques and each edge of $G$ lies in a constant number of Delsarte cliques (for example, $G$ is an edge-transitive strongly regular graph with Delsarte cliques), then any copy (as an induced subgraph) of the complete bipartite graph with parts of size $-s$ in $G$ gives rise to an eigenfunction of $G$ whose cardinality of support meets the weight-distribution bound and which is of the form given in item {\rm(1)}.\\
{\rm (4)} If the complement of $G$ has Delsarte cliques and each edge of the complement of $G$ lies in a constant number of Delsarte cliques (for example, $G$ is a coedge-transitive strongly regular graph whose complement has Delsarte cliques), then any copy (as an induced subgraph) of a pair of isolated cliques of size $1+r$ in $G$ gives rise to an eigenfunction of $G$ whose cardinality of support meets the weight-distribution bound and which is of the form given in item {\rm(2)}.
\end{lemma}
\begin{proof}
It follows from \cite[Theorem 3, Theorem 4]{KMP16} and the fact that taking the complement of a strongly regular graph preserves the eigenspaces corresponding to the non-principal eigenvalues, which describes the support of an eigenfunction of a strongly regular graph that meets the weight-distribution bound.   
\end{proof}

Thus, in view of Lemma \ref{OptimalEigenfunctionsSRG}, to show the tightness of the weight-distribution bound for non-principal eigenvalues, it suffices to find a special induced subgraph (a pair of isolated cliques $T_0$ and $T_1$ or a complete bipartite graph with parts $T_0$ and $T_1$) and show that each vertex outside of $T_0 \cup T_1$ has the same number of neighbours in $T_0$ and $T_1$. Moreover, if we can verify the condition from Lemma \ref{OptimalEigenfunctionsSRG}(3) or \ref{OptimalEigenfunctionsSRG}(4), then it suffices to find a required special induced subgraph.

\subsection{Geometric Steiner systems} \label{bsts}
Let $V(n, q)$ be a vector space of dimension $n$ over a finite field $\mathbb{F}_q$. 
The \emph{affine space} $\operatorname{AG}(n, q)$ is the geometry whose points, lines, planes, etc., are the cosets of the subspaces of $V(n, q)$ of dimension $0, 1, 2, \ldots, n$. Denote by $AS(n,q)$ the Steiner system whose points and blocks are the points and the lines of the affine space $\operatorname{AG}(n,q)$, and let us call such a Steiner system an \emph{affine Steiner system}. Such a system is known to be a 2-$(q^n, q, 1)$ design and is called an \emph{affine design} throughout the paper.

The \emph{projective space} $\operatorname{PG}(n, q)$ is the geometry whose points, lines, planes, etc., are the subspaces of $V(n+1, q)$ of dimension $1, 2, \ldots, n+1$. Denote by $PS(n,q)$ the Steiner system whose points and blocks are the points and the lines of the projective space $\operatorname{PG}(n,q)$, and let us call such a Steiner system a \emph{projective Steiner system}. 
Such a system is known to be a 2-$((q^{n+1} - 1)/(q - 1), q+1, 1)$ design and is a called a \emph{projective design} throughout the paper.
 
The \emph{block graph of a $2$-$(N, M, 1)$ design} is the graph with the blocks of the design as the vertices in which two blocks are adjacent if and only if they intersect.

\begin{theorem}[{\cite[Theorem 5.3.1]{GM15}}]\label{BlockGraphsOfDesigns}
The block graph of a $2$-$(N, M, 1)$ design (that is not symmetric)
is strongly regular with parameters    
$$
\left(\frac{N(N-1)}{M(M-1)}, \frac{M(N-M)}{M-1}), (M-1)^2+\frac{N-1}{M-1}-2,M^2
\right)$$
and smallest eigenvalue $-M$.
\end{theorem}

The following proposition gives the weight-distribution bound for the smallest eigenvalue of the block graph of a 
2-$(N,M,1)$ design.
\begin{proposition}\label{k_prop}
A  $(-M)$-eigenfunction of the block graph of a 2-$(N,M,1)$ design has at least $2M$ non-zeroes.
\end{proposition}
\begin{proof}
    It follows from Theorem \ref{BlockGraphsOfDesigns} and Corollary \ref{WDBsrg}.
\end{proof}

\subsection{Reguli in projective spaces}
In this section, we follow the definitions from \cite{BR98}.
Let $P$ be a projective space. A set $S$ of subspaces of $P$ is called \emph{skew} if no two distinct subspaces of $S$ have a point in common. Let $S$ be a set of skew subspaces. A line is called a \emph{transversal} of $S$ if it intersects each subspace of $S$ in exactly one point. 

\begin{lemma}[{\cite[Lemma 2.4.1]{BR98}}]\label{twoskewlines}
Let $\ell_1$ and $\ell_2$ be two skew lines from $P$, and denote by $t$ a point outside $\ell_1$ and $\ell_2$. Then there is at most one transversal of $\ell_1$ and $\ell_2$ through $t$. If $P$ is $3$-dimensional, then there is exactly one transversal of $\ell_1$ and $\ell_2$ through $t$. 
\end{lemma}

Let $P$ be a $3$-dimensional projective space. A nonempty set $R$ of skew lines of $P$ is called a \emph{regulus} if the following statements are true:\\
(1) Through each point of each line of $R$ there is a transversal of $R$.\\
(2) Through each point of a transversal of $R$ there is a line of $R$. 

It follows from the definition that the set $R^{opp}$ of all transversals of a regulus $R$ again forms a regulus; we call it the \emph{opposite regulus} of $R$. 

If $P$ is a $3$-dimensional projective space over a finite field $\mathbb{F}_q$, then any regulus consists of exactly $q+1$ lines.

\begin{theorem}[{\cite[Theorem 2.4.3]{BR98}}]\label{threeskewlines}
Let $P$ be a $3$-dimensional projective space over field $\mathbb{F}$. Let $\ell_1,\ell_2,\ell_3$ be three skew lines of $P$. Then there is exactly one regulus through $\ell_1,\ell_2,\ell_3$.
\end{theorem}

\begin{lemma}\label{ndimregulus} 
The following statements hold.\\
{\rm (1)} Any two skew lines in $\operatorname{PG}(n,q)$ uniquely determine a subspace $\operatorname{PG}(3,q)$, which also contains all transversals to these lines.\\
{\rm (2)} Let $U$ be an induced complete bipartite subgraph in $J_q(n,2)$. Then all the lines corresponding to the vertices of $U$ lie in a subspace $\operatorname{PG}(3,q)$, and the parts of $U$ form a regulus and an opposite regulus in this subspace.\\
\end{lemma}
\begin{proof}
(1) These two lines, being two 2-dimensional vector spaces, uniquely generate a 4-dimensional vector space, which is a subspace $\operatorname{PG}(3,q)$.

(2) It follows from item (1). 
\end{proof}

\subsection{Equitable partitions as eigenfunctions}

Let $G$ be a $k$-regular graph with the vertex set $V(G)$. Let $\Pi := (V_1,\ldots,V_t)$ be a partition of $V(G)$ into $t$ parts ($t$-partition). The partition $\Pi$ is said to be an\emph{equitable} $t$-partition if for any $i,j \in \{1,\ldots,t\}$ there is
a constant $p_{ij}$ such that any vertex from the part $V_i$ is adjacent to precisely $p_{ij}$ vertices from the part $V_j$.
The square matrix $P_\Pi:=(p_{ij})_{i,j = 1}^t$ is called the quotient matrix of the equitable $t$-partition $\Pi$.
Since all row sums of the adjacency matrix $A$ of $G$ and the quotient matrix $P_\Pi$ are equal to $k$,
both matrices have eigenvalue $k$. Moreover, it follows from \cite[Theorem 9.3.3]{GR01},
that every eigenvalue of $P$ is an eigenvalue of $A$. The eigenvalue $k$ of $A$ is called \emph{principal}.
An eigenvalue $\theta$ of $A$ is called \emph{non-principal} if $\theta \ne k$. If $\Pi$ is an equitable $2$-partition,
then precisely one non-principal eigenvalue $\theta$ is an eigenvalue of the quotient matrix $P_\Pi$.
In this case we say that the equitable $2$-partition $\Pi$ is $\theta$-\emph{equitable}.

\begin{lemma}\label{theta}
Let $\Pi$ be a $\theta$-\emph{equitable} $2$-partition of $G$ with quotient matrix
$$
P_\Pi =
\left(
  \begin{array}{cc}
   p_{11} & p_{12} \\
  p_{21} & p_{22} \\
  \end{array}
\right).
$$
Then the eigenvalues of $P_\Pi$ are given by $k = p_{11} + p_{12} = p_{21} + p_{22}$ and $\theta = p_{11}-p_{21} = p_{22}-p_{12}$.
\end{lemma}

For a sequence of distinct real values $S$, we say that a function $f$ is $S$-valued if
the image $\operatorname{Im}(f)$ coincides with the set of elements arising in $S$. Suppose a function $f:V(G) \rightarrow \mathbb{R}$ is $(x_1,x_2)$-valued.
Then the vertex set is naturally partitioned into two parts $(V_1,V_2)$ where the function $f$ takes values $x_1$ and $x_2$ on $V_1$ and
$V_2$, respectively. We denote such a $2$-partition by $\Pi_{(x_1,x_2)}(f)$. 

The following lemma, which establishes a one-to-one correspondence between $\theta$-equitable $2$-partitions and $\theta$-eigenfunctions that take
precisely two values, is well-known.
\begin{lemma}\label{1to1}
Let $G$ be a regular graph. For a partition $\Pi = (V_1,V_2)$ of the vertex set of $G$,
the following statements are equivalent.\\
{\rm (i)} The $(x_1,x_2)$-valued function $f$ such that $\Pi_{(x_1,x_2)}(f) = \Pi$ is a $\theta$-eigenfunction of $G$.\\
{\rm (ii)} The partition $\Pi$ is $\theta$-equitable and $(x_1,x_2)^T$ is an eigenvector of $P_\Pi$ corresponding to the eigenvalue $\theta$.
\end{lemma}

In view of Lemma \ref{1to1}, the investigation of equitable $2$-partitions is equivalent to the investigation of eigenfunctions
with two values.

\section{Constructions and characterisation of optimal eigenfunctions}\label{ConstructionsAndCharacterisation}
In this section we prove Theorem \ref{CharacterisationGrassmann}, Theorem \ref{CharacterisationAffine}, Proposition \ref{ConstructionOfAffineRegulus} and Proposition \ref{WDBPlus2}. 
\subsection{Proof of Theorem \ref{CharacterisationGrassmann}}
In this section we consider the block graphs of the projective Steiner system $PS(n,q)$, known as the class of strongly regular Grassmann graphs $J_q(n+1,2)$. A Grassmann graph $J_q(n+1,2)$ is known to be arc-transitive and have Delsarte cliques (the pencil of lines through a point is an example of a Delsarte clique). In view of Lemma \ref{OptimalEigenfunctionsSRG}(3), any copy (as an induced subgraph) of the complete bipartite graph with parts of size $q+1$ in $J_q(n+1,2)$ gives rise to an $-(q+1)$-eigenfunction of $J_q(n+1,2)$ whose cardinality of support meets the weight-distribution bound.
It follows from Lemma \ref{ndimregulus}(2) that the reguli in all 3-dimensional subspaces of $\operatorname{PG}(n,q)$ exhaust the induced complete bipartite subgraphs with parts of size $q+1$ in $J_q(n+1,2)$, which gives the required characterisation of optimal $-(q+1)$-eigenfunctions. This proves Theorem \ref{CharacterisationGrassmann}.

\subsection{Proof of Theorem \ref{CharacterisationAffine}, Proposition \ref{ConstructionOfAffineRegulus} and Proposition \ref{WDBPlus2}}
In this section we consider the block graphs of the affine Steiner system $AS(n,q)$, denoted by $X_q(n,1)$.  These graphs are known to be arc-transitive and have Delsarte cliques (the pencil of lines through a point is an example of a Delsarte clique). In view of Lemma \ref{OptimalEigenfunctionsSRG}(3), any copy (as an induced subgraph) of the complete bipartite graph with parts of size $q$ in $X_q(n,1)$ gives rise to an $-q$-eigenfunction of $X_q(n,1)$ whose cardinality of support meets the weight-distribution bound. 

First, consider the 3-dimensional case. Let $U$ be an induced complete bipartite subgraph in $X_q(3,1)$ with parts $\ell_{1,1},\ell_{1,2},\ldots, \ell_{1,q}$ and $\ell_{2,1},\ell_{2,2},\ldots, \ell_{2,q}$. Note that any two lines in the same part are non-intersecting and any two lines in different parts are intersecting. There are two types of pairs of non-intersecting lines in $\operatorname{AG}(3,q)$: parallel lines and skew lines.

\begin{lemma}\label{GridsInAG3q}
For any $i \in \{1,2\}$, the lines $\ell_{i,1},\ell_{i,2},\ldots, \ell_{i,q}$ are either pairwise parallel, forming a parallel class in a subplane $\operatorname{AG}(2,q)$, or pairwise skew.   
\end{lemma}
\begin{proof}
Suppose two of the lines are parallel. Then they lie in the same subplane together with all transversals. This means that all lines $\ell_{i,1},\ell_{i,2},\ldots, \ell_{i,q}$ lie in the same plane.   
\end{proof}

Let $A$ be a $3$-dimensional affine space. A nonempty set $R$ of pairwise skew lines of $A$ is called an \emph{affine regulus} if the following statements are true:\\
(1) Through each point of each line of $R$ there is a transversal of $R$.\\
(2) Through each point of a transversal of $R$ there is a line of $R$. 

It follows from the definition that the set $R^{opp}$ of all transversals of an affine regulus $R$ again form an affine regulus; we call it the \emph{opposite affine regulus} of $R$. 
The following proposition shows that affine reguli in $\operatorname{AG}(3,q)$ necessarily come from projective reguli in $\operatorname{PG}(3,q)$.

\begin{proposition}\label{CharacterisationOfAffineReguli}
The following statements hold.\\
{\rm(1)} Let $R$ and $R^{opp}$ be a regulus in $\operatorname{PG}(3,q)$ and its opposite. Let $\ell_1 \in R, \ell_2 \in R^{opp}$ be two lines and $\Pi$ be the plane through these lines. Let $R'$ and $(R^{opp})'$ be two sets of $q$ affine lines in the affine space $\operatorname{PG}(3,q) \setminus \Pi$ remaining from $R$ and $R^{opp}$. Then $R'$ and $(R^{opp})'$ are an affine regulus and its opposite. \\
{\rm(2)} Every affine regulus in $\operatorname{AG}(3,q)$ comes from a projective regulus in the way described in item (1).
\end{proposition}
\begin{proof}
(1) This is straightforward.

(2) If $q = 2$, the result follows trivially. Suppose that $q > 2$. Let $\ell_1,\ell_2,\ell_3$ be a triple of pairwise skew lines in $\operatorname{AG}(3,q)$.

\textbf{Case 1:} These three lines lie in the same class of parallel planes, occupying three planes.

\textbf{Case 2:}
Each two of the three lines define their own class of parallel planes.

In Case 1, the three points of intersection with the plane at infinity lie on the same line at infinity. In Case 2, the three intersection points do not lie on the same line in the plane at infinity.
Then consider the projectivisation of the affine space. Each of the three lines $\ell_1,\ell_2,\ell_3$ is completed by a point from the plane at infinity. In the projective space we still have a triple of skew lines $\ell_1',\ell_2',\ell_3'$. 
Consider the projective regulus $R$ uniquely determined by the triple $\ell_1',\ell_2',\ell_3'$ of skew lines. 
It is clear that the restriction of $R$ to the
affine part is an affine regulus if and only if one line of the opposite regulus $R^{opp}$ is in the plane at infinity, that is, if and only if the lines $\ell_1,\ell_2,\ell_3$ are from Case 1. Thus, in Case 1 the lines $\ell_1,\ell_2,\ell_3$ can be uniquely extended to an affine regulus. 

To complete the proof, let us show that if $\ell_1,\ell_2,\ell_3$ is a triple of lines from Case 2, then this triple cannot be extended to an affine regulus. Suppose to the contrary there exist lines $m_1,\ldots,m_{q-3}$ such that $\ell_1,\ell_2,\ell_3,m_1,\ldots,m_{q-3}$ is an affine regulus. Consider the projective regulus $\ell_1',\ell_2',\ell_3',k_1',\ldots,k_{q-2}'$ uniquely determined by the triple $\ell_1',\ell_2',\ell_3'$. Note that each transversal to this projective regulus does not lie in the plane at infinity (otherwise, the triple $\ell_1,\ell_2,\ell_3$ would be from Case 1), that is, each such a transversal intersects the plane at infinity at one point. Note that the restriction of each such a transversal to the affine part is a transversal to the affine regulus. Thus, the affine regulus has at least $q+1$ transversals. By the pigeonhole principle, there exists a point of the affine regulus incident to two transversals, which means the entire affine regulus lies in the affine plane determined by these two transversals, a contradiction.

\end{proof}

Now the statement of Theorem \ref{CharacterisationAffine} in the 3-dimensional case follows from Lemma \ref{GridsInAG3q} and Proposition \ref{CharacterisationOfAffineReguli}. The $n$-dimensional case of Theorem \ref{CharacterisationAffine} follows from the next lemma.

\begin{lemma}\label{ndimregulusAffine} 
The following statements hold.\\
{\rm (1)} Any two skew lines in $\operatorname{AG}(n,q)$ uniquely determine a subspace $\operatorname{AG}(3,q)$, which also contains all transversals to these lines.\\
{\rm (2)} Let $U$ be an induced complete bipartite subgraph in $X_q(n,1)$, where $n \ge 3$. Then all the lines corresponding to the vertices of $U$ lie in a subspace $\operatorname{AG}(3,q)$, and the parts of $U$ form an affine regulus and an opposite affine regulus in this subspace.\\
\end{lemma}
\begin{proof}
(1) Let $\ell_1$ and $\ell_2$ be two skew lines in $\operatorname{AG}(n,q)$. Without loss of generality, we may assume that $\ell_1$ contains zero. Let $\ell_2'$ be the uniquely determined line that is parallel to $\ell_2$ and passes through zero. Let $\ell_3$ be a line connecting zero with some point of $\ell_2$. Let $v_1,v_2,v_3$ be vectors generating the lines $\ell_1,\ell_2',\ell_3$ through zero, respectively, and let $V$ be the subspace generated by $v_1,v_2,v_3$. Then $V$ forms a 3-dimensional subspace in $\operatorname{AG}(n,q)$ and $V$ contains the lines $\ell_1,\ell_2$ and all transversals to them.

(2) It follows from item (1) and Proposition \ref{CharacterisationOfAffineReguli}. 
\end{proof}

The proof of Proposition \ref{ConstructionOfAffineRegulus} is direct: one can show that any two lines from the same set are skew and any two lines from different sets intersect. Let $V, V_1$ and $V_2$ be the subspace generated by $v_1,v_2,v_3$, the subspace generated by $v_1,v_3$ and the subspace generated by $v_2,v_3$, respectively. It is easy to see that all lines from $S_1 \cup S_2$ belong to $V$. One can also show that the class of parallel planes in $V$ given by the plane $V_1$ (resp. $V_2$) contains the lines from $S_1$ (resp. $S_2$).

The proof of Proposition \ref{WDBPlus2} follows from the definition of an eigenfunction.

\section{Proof of Proposition \ref{balanceTheorem} and Corollary \ref{Balance}}\label{sec:balance}
Let us give a proof of Proposition \ref{balanceTheorem}.

Consider the scalar product of the orthogonal functions $f_1$ and $f_2$.
Since $f_1$ is $(1,-1,0)$-valued, this gives
$$
\sum\limits_{u \in Supp^+(f_1)}f_2(u) - \sum\limits_{u \in Supp^-(f_1)}f_2(u) = 0,
$$
and, consequently,
$$
\sum\limits_{u \in Supp^+(f_1)}f_2(u) = \sum\limits_{u \in Supp^-(f_1)}f_2(u).
$$

Put $m^+:=\overline{Supp^+(f_1)}$ and $m^-:=\overline{Supp^-(f_1)}$. By definitions and the fact that $f_2$ is $(x_1,x_2)$-valued,
we get
$$
m^+\cdot x_1 + (|Supp^+(f_1)| - m^+)\cdot x_2 = m^-\cdot x_1 + (|Supp^-(f_1)| - m^-)\cdot x_2.
$$
and then, taking into account that $|Supp^+(f_1)| = |Supp^-(f_1)|$,
$$
(m^+-m^-)(x_1 - x_2) = 0.
$$
Since $x_1 \ne x_2$, we finally conclude that $m^+ = m^-$, which completes the proof of Proposition \ref{balanceTheorem}.

Corollary \ref{Balance} is then straightforward.

\section*{Acknowledgements} \label{Ack}
D.~Panasenko is supported by Russian Science Foundation according to the research project 22-21-20018. S. Goryainov is supported by the Special Project on Science and Technology Research and Development Platforms, Hebei Province (22567610H).
The authors are grateful to Denis Krotov for the discussions concerning the results in this paper. Finally, the authors thank two anonymous referees whose valuable comments improved the content of the paper. 


\end{document}